\documentclass[12pt]{amsart}

\usepackage{amsmath} 
\usepackage{amssymb} 
\usepackage{amsthm}
\usepackage{amscd}
\newtheorem{thm}{Theorem}[section]
\newtheorem{lem}[thm]{Lemma}
\newtheorem{lem-dfn}[thm]{Lemma-Definition}
\newtheorem{prop}[thm]{Proposition}
\newtheorem{cor}[thm]{Corollary}

\theoremstyle{definition}
\newtheorem{defn}[thm]{Definition}
\newtheorem{exam}[thm]{Example}
\newtheorem{ex}[thm]{Example}

\newtheorem{quest}[thm]{Question}

\newtheorem*{acknowledgement}{Acknowledgement}
\theoremstyle{remark}

\newtheorem{rem}[thm]{Remark}
\numberwithin{equation}{section}

\newcommand{\thmref}[1]{Theorem~\ref{#1}}
\newcommand{\lemref}[1]{Lemma~\ref{#1}}
\newcommand{\corref}[1]{Corollary~\ref{#1}}
\newcommand{\proref}[1]{Proposition~\ref{#1}}
\newcommand{\remref}[1]{Remark~\ref{#1}}
\newcommand{\queref}[1]{Question~\ref{#1}}

\newcommand{\defref}[1]{Definition~\ref{#1}}

\newcommand{\exref}[1]{Example~\ref{#1}}

\newcommand{\sref}[1]{Section~\ref{#1}}

\DeclareMathOperator{\Spec}{Spec}
\DeclareMathOperator{\spec}{Spec}

\DeclareMathOperator{\supp}{Supp}

\DeclareMathOperator{\chara}{char}

\DeclareMathOperator{\di}{div}

\DeclareMathOperator{\core}{core} 
\DeclareMathOperator{\chr}{char} 
\newcommand{\m}{\mathfrak m}

\newcommand{\frm}{\mathfrak{m}}

\newcommand{\PP}{\mathbb P}

\newcommand{\Z}{\mathbb Z}
\newcommand{\Q}{\mathbb Q}

\newcommand{\cal}{\mathcal}

\newcommand{\cL}{\mathcal L}

\newcommand{\cO}{\mathcal O}

\newcommand{\cR}{\mathcal R}

\renewcommand{\:}{\colon}

\newcommand{\ol}[1]{\overline {#1}}

\newcommand{\defset}[2]{{\left\{#1\,\left| \,#2 \right. \right\}}}

\begin{document}

\title[A characterization of two-dimensional rational singularities]{A characterization of two-dimensional rational singularities via Core of ideals}

\author{Tomohiro Okuma}
\address[Tomohiro Okuma]{Department of Mathematical Sciences,
Faculty of Science, Yamagata University, Yamagata, 990-8560, Japan.}
\email{okuma@sci.kj.yamagata-u.ac.jp}
\author{Kei-ichi Watanabe}
\address[Kei-ichi Watanabe]{Department of Mathematics, College of
Humanities and Sciences,
Nihon University, Setagaya-ku, Tokyo, 156-8550, Japan}
\email{watanabe@math.chs.nihon-u.ac.jp}
\author{Ken-ichi Yoshida}
\address[Ken-ichi Yoshida]{Department of Mathematics,
College of Humanities and Sciences,
Nihon University, Setagaya-ku, Tokyo, 156-8550, Japan}
\email{yoshida@math.chs.nihon-u.ac.jp}

\thanks{This work was partially supported by JSPS  KAKENHI 
Grant Numbers 26400064, 26400053, 25400050}

\begin{abstract}
The notion of $p_g$-ideals for normal surface singularities has been proved to be very useful. 
On the other hand, the core of ideals has been proved to be very 
important concept and also very mysterious one.  
However, the computation of the core of an ideal seems to be given only for very special cases.  
In this paper, we will give an explicit description of the core of $p_g$-ideals of normal surface singularities. 
 As a consequence, we give a characterization of rational singularities
  using the inclusion of the core of integrally closed ideals.
\end{abstract}

\maketitle

\section{Introduction}
Let $(A,\m)$ be a two-dimensional excellent 
normal local domain containing an algebraically closed field. 
We always assume that $(A,\m)$ is not regular.
When $(A,\m)$ is a rational singularity, 
 Lipman \cite{Li} proved that any integrally closed $\m$-primary ideal $I$ is stable, namely, $I^2=QI$ for some (every) minimal reduction $Q$, and that if $I$ and $I'$ are integrally closed  $\m$-primary ideals, then the product $II'$ is also integrally closed.
(Later, Cutkosky \cite{CharRat} showed that this property characterizes
the rational singularities for two-dimensional excellent normal local domains.)
These facts play very important role to study ideal theory on a two-dimensional rational singularity.

\par 
In \cite{OWYgood}, the authors introduced the notion of $p_g$-ideals
for two-dimensional excellent normal local domain containing 
 an algebraically closed field and proved
that the $p_g$-ideals inherit nice
properties of integrally
closed ideals of rational singularities (in a rational singularity,
every integrally closed
ideal is a $p_g$-ideal by our definition).
Namely, any $p_g$-ideal $I$ is stable and if $I$ and $I'$ 
are $p_g$-ideals, then $II'$ is integrally closed and
 also a $p_g$-ideal.

Let $ f: X \to \Spec A$ be a resolution of singularity.
Then $p_g(A) := \ell_A(H^1(X,\cO_X))$ is independent of the choice of a resolution and an important invariant of $A$
(here we denote by $\ell_A(M)$ the length of an $A$ module $M$).
The invariant $p_g(A)$ is called the {\em geometric genus} of $A$.
A rational singularity is characterized by $p_g(A)=0$.
Let $I \subset A$ be an integrally closed $\m$-primary ideal. 
Then there exists a resolution of singularities $f \colon X \to \Spec A$ and an 
anti-nef cycle $Z$ on $X$ so that $I\mathcal{O}_X=\cO_X(-Z)$ and
$I=I_Z:=H^0(X, \mathcal{O}_X(-Z))$. 
In general, we can show that $\ell_A(H^1(X, \mathcal{O}_X(-Z)))\le p_g(A)$ for any cycle $Z$ such that 
$\mathcal{O}_X(-Z)$ has no fixed components.  
If equality holds, then $Z$ is called a
{\em $p_g$-cycle} and $I=I_Z$ is called a {\em $p_g$-ideal}.

The {\em core} of an ideal $I$ is the intersection of all reductions of $I$. 
The notion of core of ideals was introduced by Rees and Sally \cite{Rees-Sally}, and many properties of the core have been shown by 
Corso--Polini--Ulrich \cite{CPU2001}, \cite{CPU2002}, Huneke--Swanson \cite{HS-core}, Huneke--Trung \cite{HT-core}, 
Hyry--Smith \cite{Hyry-Smith}, Polini--Ulrich \cite{PU2005}.
The core of ideals is related to coefficient, adjoint and multiplier ideals.
However, it seems to the authors that the computation of the core is given only for very special cases.  

In this paper, we will show that if 
$I=I_Z$ is a $p_g$-ideal and $Q$ is a minimal reduction of $I$, 
there exists a cycle $Y\ge 0$ on $X$ such that $Q:I=I_{Z-Y}$ and $\core(I)=I_{2Z-Y}$.
In particular, this implies that if $A$ is rational, $X_0$ is the minimal resolution, and $K = K_{X/{X_0}}$ (the relative canonical divisor), then $\core(I) = I_{2Z-K}$
(see \thmref{algorithm} for general case).
Moreover, we obtain that if $I$ and $I'$ are $p_g$-ideals and $I'\subset I$, then $\core(I') \subset \core(I)$ (see \thmref{t:subset}).
If $A$ is not a rational singularity, then there exist a $p_g$-ideal $I'$ and an integrally closed non-$p_g$-ideal $I$ such that
$I'\subset I$ and $\core(I') \not\subset \core(I)$ (see \proref{C-E}).  
Therefore, we obtain a characterization of rational singularities in terms of core of ideals, which is our main theorem.

\begin{thm}[See \thmref{mainThm}]\label{IntMain}
The following conditions are equivalent if $\chara(k)\ne 2$ 
\begin{enumerate} 
\item For any integrally closed $\frm$-primary ideals $I'\subset I$, we have $\core(I')\subset \core(I)$.
\item $A$ is a rational singularity.
\end{enumerate}
\end{thm}

An $\m$-primary ideal $I$ is said to be {\em good} 
if $I^2=QI$ and $I=Q: I$ for some minimal reduction $Q$ of $I$ (see 
\cite{GdId}, \cite{OWYgood}).  
As an application of the theorem, we show the following existence theorem for good ideals, which generalizes 
\cite[Theorem 4.1]{OWYgood} for non-Gorenstein local domains.  

\begin{thm}[See {\thmref{algorithm}}] \label{Main2}
Any two-dimensional excellent normal local domain $A$ over an algebraically closed field 
admits a good $p_g$-ideal. 
\end{thm}

\par 
Let us explain the organization of this paper. 
\par 
In \sref{s:Pre}, we recall the definition and several basic properties of $p_g$-ideals. 
In \sref{s:quest}, we prove \thmref{IntMain} using the property that taking the core preserves the inclusion of $p_g$-ideals 
(\thmref{t:subset}), 
and will show that if $I_0$ is not a $p_g$-ideal,  then for some $n\ge 1$ we can 
construct a $p_g$-ideal $I' \subset I= \overline{I_0^n}$ with 
 $\core(I') \not\subset \core(I)$.  In \sref{s:comp}, we show that $Q: I$ and $\core(I)$ are also $p_g$-ideals for a $p_g$-ideal $I$ and its minimal reduction $Q$
(\thmref{Q:I}), and give an algorithm for obtaining these ideals (\thmref{algorithm}).
We also prove some characterization of good ideals and the existence of good ideals in $A$.  
In \sref{s:pgexs}, we give a maximal $p_g$-ideal contained in a given integrally closed $\m$-primary ideal.

\section{Preliminary on $p_g$-ideals}\label{s:Pre}

Throughout this paper, let $(A,\m)$ be a 
two-dimensional excellent normal local domain containing an algebraically closed field and 
$f\:X \to \spec A$ a resolution of singularities 
with exceptional divisor $E:=f^{-1}(\m)$ unless otherwise specified. 
Let $E=\bigcup_{i=1}^rE_i$ be the decomposition into  irreducible components of $E$.
A $\Z$-linear combination $Z=\sum_{i=1}^r n_iE_i$ is called a {\it cycle}.
We say a cycle $Z=\sum_{i=1}^r n_iE_i$ is {\it effective} if $n_i\ge 0$ for every $i$.
We denote $Z\ge Z'$ if $Z-Z'$ is effective.

\par \vspace{2mm}
First, we recall the definition of $p_g$-ideals. 
If $I$ is an integrally closed $\m$-primary ideal of $A$, then 
there exists a resolution $X \to \Spec A$ and an effective cycle $Z$ on $X$ 
such that $I \cO_X = \cO_X(-Z)$.
In this case, we denote the ideal $I$ by $I_Z$,\footnote{When we write $I_Z$, we assume that $\cO_X(-Z)$ is generated by global sections.} 
and we say that $I$ is  {\em represented on} $X$ by $Z$.
Note that  $I_Z=H^0(X,\cO_X(-Z))$.
Furthermore, if $ZC<0$ for every $(-1)$-curve $C$ on $X$, then we say that $I$ is {\em minimally} represented.
We obtain a unique minimal representation of $I$ by contracting 
$(-1)$-curves $C$ 
with $ZC=0$ successively (cf. \cite[\S 3]{OWYgood}).
\par  
We say that $\cO_X(-Z)$ {\em has no fixed component}
if $H^0(X, \cO_X(-Z))\ne H^0(X, \cO_X(-Z-E_i))$ for every $E_i\subset E$, 
i.e., the base locus of the linear system $H^0(X, \cO_X(-Z))$ does not contain any component of $E$.
It is clear that $\cO_X(-Z)$ has no fixed component when $I$ is represented by $Z$.
\par 
We denote by $h^1(\cO_X(-Z))$ the length $\ell_A(H^1(X,\cO_X(-Z)))$. 
We put $p_g(A) = h^1(\cO_X)$ and  call it the \textit{geometric genus} of $A$.

The number $h^1(\cO_X(-Z))$ is independent of the choice of representations of $I=I_Z$ (cf. \cite{OWYgood}) and 
is an important invariant of the ideal for our theory.
We have the following result for this invariant.

\begin{prop}[{\cite[2.5, 3.1]{OWYgood}}]\label{p:le-cycle}
Let $Z'$ and $Z$ be  cycles on $X$ and assume that $\cO_X(-Z)$ has no fixed components.
Then we have $h^1(\cO_X(-Z'-Z)) \le  h^1(\cO_X(-Z'))$.  
 In particular, we have $h^1(\cO_X(-Z))\le p_g(A)$; if the equality holds, then $\cO_X(-Z)$ is generated.
\end{prop}

\begin{defn}\label{p_g-dfn}
(1) Let $I=I_Z= H^0(X,\cO_X(-Z))$ be an integrally closed ideal represented by $Z>0$ on $X$.
We call $I$ a {\it $p_g$-ideal} if $h^1(\cO_X(-Z))=p_g(A)$.  

(2) A cycle $Z>0$ is called a {\em $p_g$-cycle} if $\cO_X(-Z)$ is generated and $h^1(\cO_X(-Z))=p_g(A)$.
\end{defn}

\begin{rem}
If $p_g(A)=0$, then  $A$ is called a {\em rational} singularity. 
On a rational singularity, every integrally closed $\m$-primary ideal is a 
$p_g$-ideal by definition (cf. \cite{Li}) and conversely, this property characterizes  a rational singularity
because we always have integrally closed ideal $I=I_Z$ such that $h^1(\cO_X(-Z))=0$ (cf. \cite{Gi}).
\end{rem}

In \cite{OWYgood}, we have seen many good properties of $p_g$-ideals.
We will review some of these properties.

Let $Z$ and $Z'$ be nonzero effective cycles on the resolution $X \to \Spec A$
such that $\cO_X(-Z)$ and $\cO_X(-Z')$ are generated. 
Take general elements $a \in I_{Z}$, $b \in I_{Z'}$ and put 
\begin{eqnarray*}
\varepsilon(Z,Z')
&:= & \ell_A(I_{Z+Z'}/a I_{Z'}+bI_{Z}) \\
& = & p_g(A)-h^1(\cO_X(-Z))-h^1(\cO_X(-Z'))+h^1(\cO_X(-Z-Z')). 
\end{eqnarray*}
Then $0 \le \varepsilon(Z,Z') \le p_g(A)$; see \cite[2.6]{OWYgood}. 

\par 
The following proposition gives an important property of $p_g$-ideals. 
\begin{prop}[\textrm{see \cite[3.5, 3.6]{OWYgood}}]\label{p:sg}
Let $Z$ and $Z'$ be the cycles as above. 
\begin{enumerate}
\item 
If $Z$ is a $p_g$-cycle,  
then $\varepsilon(Z, Z') =0$ for any $Z'$, and 
\[
I_{Z+Z'}=a I_{Z'}+b I_{Z}
\]
for general elements $a \in I_{Z}$ and $b \in I_{Z'}$.
In particular, the product $I_ZI_{Z'}$ is integrally closed and 
$h^1(\cO_X(-Z-Z')) = h^1(\cO_X(-Z'))$. 

\item $Z$ and $Z'$ are $p_g$-cycles if and only if so is $Z+Z'$. 
\end{enumerate}
\end{prop}

\begin{proof}
The ``if part'' of (2) follows from \proref{p:le-cycle}.
The other claims follow from 
\cite[3.5]{OWYgood}.
\end{proof}

Recall that an ideal $J \subset I$ is called a {\it reduction} of $I$ if $I$ is integral over $J$ or, 
equivalently,  $I^{r+1} = I^rJ$ for some $r$ (e.g., \cite{SH}). 
An ideal  $Q\subset I$ is called a 
{\it minimal reduction} of $I$ if $Q$ is minimal among the reductions of $I$. 
Any minimal reduction of an $\m$-primary ideal is a parameter ideal.

\begin{cor}[\textrm{see \cite[3.6]{OWYgood}}]\label{c:sg}
Let  $I$ and $I'$  be any integrally closed $\m$-primary ideals.
\begin{enumerate}
\item $I$ and $I'$  are $p_g$-ideals if and only if so is $II'$. 
\item Assume that $I$ is a $p_g$-ideal.
Then $I^n$ is a $p_g$-ideal for every $n>0$. 
If  $Q$ is a minimal reduction of $I$,  then $I^2=QI$ and $I\subset Q: I$.
\end{enumerate}

\end{cor}

Assume that $p_g(A)> 0$. 
Then we give a characterization of $p_g$-ideals by cohomological cycle.  
Let $K_X$ denote the canonical divisor on $X$.
Let $Z_{K_X}$ denote the canonical cycle, 
i.e., the $\Q$-divisor  supported in $E$ 
such that $K_X+Z_{K_X}\equiv 0$.
By \cite[\S 4.8]{chap}, 
there exists the smallest cycle $C_X>0$ on $X$ such that $h^1(\cO_{C_X}) =p_g(A)$, and $C_X=Z_{K_X}$ if $A$ is Gorenstein and the resolution $f \colon X \to \Spec A$ is minimal.
The cycle $C_X$ is called the \textit{cohomological cycle} on $X$.

\begin{prop}\label{p:CC}
Assume that $p_g(A)>0$ and 
let $C\ge 0$ be the minimal cycle such that 
$H^0(X\setminus E, \cO_X(K_X))=H^0(X,\cO_X(K_X+C))$.
Then $C$ is the cohomological cycle.
Therefore 
if $g\:X'\to X$ is the blowing-up at a point in $\supp C_X$ and $E_0$ the exceptional set for $g$, then $C_{X'}=g^*C_X-E_0$.
For any cycle $D>0$ without common components with $C_X$, we have $h^1(\cO_D)=0$. 
\end{prop}
\begin{proof}
By the Grauert-Riemenschneider vanishing theorem\footnote{Note that 
the 
Grauert-Riemenschneider vanishing theorem holds 
in any characteristic  in dimension $2$ (\cite{Gi}).}  and  
the duality theorem, we have
\begin{align*}
\ell_A(H^0(X\setminus E, \cO_X(K_X))/H^0(X,\cO_X(K_X)))
&=\ell_A(H_E^1(X, \cO_X(K_X)))=p_g(A), \\
\ell_A(H^0(X, \cO_X(K_X+C))/H^0(X,\cO_X(K_X)))
&=h^0(\cO_C(K_X+C)))=h^1(\cO_C).
\end{align*}
Thus $h^1(\cO_C)=p_g(A)$.\footnote{This result was already obtained by Tomari \cite{tomari.tagajo}.}
Let $C'>0$ be a cycle such that $C'<C$. 
By the assumption, we have
$$
h^1(\cO_{C'})=\ell_A(H^0(\cO_X(K_X+C'))/H^0(\cO_X(K_X)))
<\ell_A(H_E^1(\cO_X(K_X)))=p_g(A).
$$

Let $D>0$ be a cycle without common components with $C$.
As in the proof of \cite[\S 4.8]{chap} (putting $A=0$),
we obtain the surjection 
$H^1(\cO_{C+D})\to H^1(\cO_{C})\oplus H^1(\cO_{D})$.
Since $h^1(\cO_{C+D})=h^1(\cO_C)=p_g(A)$, it follows that $H^1(\cO_D)=0$.
\end{proof}

We have the following characterization of $p_g$-ideals  
in terms of cohomological cycle.

\begin{prop} [{\cite[3.10]{OWYgood}}] \label{p:CX} 
Assume that $p_g(A)>0$.
Let $Z>0$ be a cycle such that $\cO_X(-Z)$ has no fixed component.
Then $Z$ is a $p_g$-cycle if and only if
$\cO_{C_X}(-Z)\cong \cO_{C_X}$.
\end{prop}

Now we will define the core of an ideal $I$ of $A$. 
 
\begin{defn}\label{core-dfn} Let $I$ be an $\m$-primary ideal of $A$.
\begin{enumerate}
\item  The {\em core} of $I$ is defined by 
\[
\core(I)=\bigcap_{\text{$Q$ is a reduction of $I$}} Q.
\]
\item  $I$ is said to be {\it stable} if $I^2=QI$ for any minimal reduction $Q$ of $I$.
By this definition, if $I$ is stable, then $I\subset Q:I$ and $I^2 \subset \core(I)$. 
\item $I$ is said to be {\it good} if  $I$ is stable and $Q:I=I$ for any minimal reduction $Q$ of $I$ (\cite{GdId}, \cite{OWYgood}).  If $I$ is integrally closed and stable, then $I^2 = \core(I)$ if and only if $I$ is a good ideal 
 (see \lemref{l:core}). 
\end{enumerate}
\end{defn}

\begin{rem}[{\cite{GdId}}]\label{r:Ggood}
 If $A$ is Gorenstein, then $I$ is good if and only if $I$ is stable and  $e(I) = 2\ell_A(A/I)$, where $e(I)$ denotes the multiplicity of $I$ (if $I=I_Z$, 
then $e(I_Z) = -Z^2$). 
We have shown the existence of good $p_g$-ideals when $A$ is Gorenstein in \cite{OWYgood}.
\end{rem}

We also have a characterization of $p_g$-ideal in terms of the Rees algebras.

\begin{prop}[{\cite{OWYrees}}]\label{p:pgC}
Let $I$ be an integrally closed $\m$-primary ideal $I$.
Then the following conditions are equivalent.
\begin{enumerate}
 \item $I$ is a $p_g$-ideal.
 \item The Rees algebra $\bigoplus_{n\ge 0}I^nt^n\subset A[t]$ is a Cohen-Macaulay normal domain. 
\item $I^2=QI$ and $\overline{I^n}=I^n$ for every $n>0$, where $\overline{J}$ denotes the integral closure of an ideal $J$.
 \end{enumerate}  
\end{prop}

In \sref{s:comp}, we will give the description of $\core(I)$ for a $p_g$-ideal $I$.

\section{A question on core of integrally closed ideals}\label{s:quest}

The core of an ideal $I$ is the intersection of all the 
reductions of $I$, or, equivalent to say, the intersection of all the minimal reductions of $I$.
We are interested whether taking the core preserves the inclusion of 
ideals.
But it is obvious that if ideals $I'\subset I$ have the same integral closure, then 
$\core(I') \supset \core(I)$, since a reduction of $I'$ is also a reduction of $I$.
So, we ask the following question.

\begin{quest}\label{Q1} Let $(A,\frm)$ be a two-dimensional excellent normal local domain and let 
$I'\subset I$ be  integrally closed $\m$-primary ideals. Then is it always true that 
$\core(I')\subset \core(I)$?
\end{quest}

 \queref{Q1} was asked in \cite{HS-core} and then in \cite[Question 5.5.2]{Hyry-Smith}.
  When $A$ is two-dimensional Gorenstein rational singularity this question was answered affirmatively (see \cite[5.5.1]{Hyry-Smith}). 
On the other hand, Kyungyong Lee \cite{Lee-core} gave a counterexample to this question in a regular local ring of dimension 4. 

The answer to this question and the main theorem of this paper is the following.

\begin{thm}\label{mainThm} 
Let $(A,\frm)$ be a two-dimensional excellent normal local domain. Then the following conditions are equivalent 
if $\chara(k)\ne 2$. 
\begin{enumerate} 
\item For any integrally closed $\frm$-primary ideals $I'\subset I$, we have $\core(I')\subset \core(I)$.
\item $A$ is a rational singularity.
\end{enumerate}
\end{thm}
\begin{proof} If $A$ is a rational singularity then every integrally closed ideal is a $p_g$-ideal.
We will show in \thmref{t:subset}  that the answer to Question \ref{Q1} is positive if both $I$ and $I'$ are 
$p_g$-ideals. Hence 
 the implication (2) $\Longrightarrow$ (1) is true.\par
 
The implication (1) $\Longrightarrow$ (2) follows from \proref{C-E}.
\end{proof}  

\begin{rem} The implication (2) $\Rightarrow$ (1) holds if $\chara(k)=2$, 
since \thmref{t:subset} does not depend on $\chara(k)$. 
\end{rem}

By applying some results in \sref{s:comp} and \sref{s:pgexs}, we shall prove 
the following.
  
 \begin{prop}\label{C-E}
Assume that $p_g(A)>0$ and  $\chr(k)\ne 2$. Let $I=I_Z$ be an integrally closed ideal
 such that $I^2\ne QI$ for some minimal reduction
$Q$ of $I$ and that $I^2$ is integrally closed.
Then there exists a $p_g$-ideal $I'\subset I$
such that $\core(I') \not\subset \core( I )$. \par
In particular, if $I=I_Z$ is an integrally closed ideal which is
{\em not} a $p_g$-ideal, then for some $n>0$,
there exists a $p_g$-ideal $I'\subset \overline{I^n}$
such that $\core(I') \not\subset \core( \overline{I^n} )$.
\end{prop}

\begin{proof} 
Suppose that a minimal reduction $Q$ of $I$ is generated by a system of parameters $a,b\in I$. 
Let us prove that $I^2\not\subset Q$.
If $I^2 \subset Q$ holds, then 
$I^2 = \overline{I^2} = \overline{I^2} \cap Q=Q\overline{I}=QI$ 
by Huneke--Itoh theorem (see \cite{Hu}, \cite{Itoh}). 
This contradicts the assumption. 
Hence $I^2 \not \subset Q$.  

%

\par
Now, since  $I^2\not\subset Q$  
 and since $\chara(k)\ne 2$, 
 for a general element $f$ of $I$, $f^2\not\in Q$ and 
in particular, $f^2\not\in\core(I)$. 
On the other hand, by \thmref{t:pgf}, there exists a $p_g$-ideal $I'\subset I$ such that $f\in I'$.
Then $f^2\in I'^2\subset \core(I')$,
 since $I'$ is stable.
This shows that $\core(I')\not\subset \core(I)$.

 For a given $I= I_Z$, which is not a $p_g$-ideal, choose $n$ so that 
$\overline{I^{nk}}=\overline{I^n}^k$ and
$h^1(\cO_X(-nZ))=h^1(\cO_X(-nkZ))$ 
for every positive integer $k$.
 Such $n$ exists since the normalized Rees algebra of $I$ is Noetherian and $h^1(\cO_X(-nZ))$ is a decreasing function of $n$ (cf. \proref{p:le-cycle}).  
Clearly $\overline{I^n}=I_{nZ}$ is not a $p_g$-ideal 
by \proref{p:sg} (2).  
Let $Q_n$ be a minimal reduction of $\overline{I^n}$. 
Since $h^1(\cO_X(-nZ))=h^1(\cO_X(-2nZ))$, from the exact sequence
\[
0 \to  \cO_X  \to  \cO_X(-nZ)^{\oplus 2}  \to \cO_X(-2nZ)  \to 0, 
\]
we obtain $\ell_A(\overline{I^n}^2/Q_n\overline{I^n})=p_g(A)-h^1(\cO_X(-nZ))>0$.
Therefore we can apply the first claim to the ideal $ \overline{I^n}$.
\end{proof}  

\begin{exam}\label{ex:244}  Let $A = k[[x,y,z]]/(x^2+ y^4+ z^4)$, where $k$ is an algebraically closed 
field.\footnote{If $\chr(k)=2$, then we change $ x^2+ y^4+ z^4$ to $x^2+ g(y,z)$, where $g(y,z)$ is a form of degree $4$ with no multiple roots.} 
Then $\core(\m) = \m^2$ since $\m$ is a good ideal (cf. \remref{r:Ggood}, \lemref{l:core}).  Note that $\core(\m)$ is {\it not} integrally closed since $x$ is integral over $\m^2$. It is easy to show that 
for every integrally closed ideal $I'\subset \m$, $\core(I')\subset \core(\m)=\m^2$.

Let $f_0 : X_0 \to \Spec A$ be the   minimal resolution.
Then 
$f_0^{-1}(\m) = E_0$ is an elliptic curve $\{x^2+ y^4+ z^4=0\}$ in the weighted projective 
space $\PP(2,1,1)$. 
We have $p_g(A)=1$, $\m=I_{E_0}$, $h^1(\cO_{X_0}(-nE_0))=0$ for $n>0$ by the vanishing theorem (cf. \cite{Gi}).
Now, let $I :=I_{2E_0}= \overline{\m^2}$, which is generated by $\m^2$ and $x$. Then $I^n$ is integrally closed (cf. \cite[4.24]{chap}) and not a $p_g$-ideal for every $n>0$.  
It is easy to see that $\ell_A(I^2/QI)=p_g(A)=1$ for any minimal reduction $Q$ of $I$.
\par 
Take a general element $g$ of $I$ and let $P_1,\ldots , P_4\in E_0$ be the intersection points of $E_0$ and $\di_X(g)-2E_0$.
Let $f_1: X \to X_0$ be the blowing-up of these $4$ points and let $E_i = f_1^{-1}(P_i)$, $i=1, \ldots ,4$.
We again denote by $E_0$ the strict transform of $E_0$ on $X$.
Then $Z = 2 E_0 + 3\sum_{i=1}^4 E_i$ is a $p_g$-cycle on $X$ and $K_XZ=0$. 
By Kato's Riemann-Roch formula (\cite{kato}, cf. \cite[2.8]{OWYgood}), we have $\ell_A( A/ I_Z)=6$ and $\ell_A( A/ \overline{\m^3})=7$. 
Thus $I_Z = (g) + \overline{\m^3}\subset I$. 
Since $2 \ell_A( A/ I_Z)=-Z^2$,
$I_Z$ is also a good ideal (cf. \remref{r:Ggood}) and $\core(I_Z) = I_Z^2$.  On the other hand, as we have seen in the proof of
Proposition \ref{C-E}, $g^2\not\in \core(I)$. Thus we conclude that  
$\core(I_Z)\not\subset \core(I)$.
\end{exam}

\section{Computation of $\core(I)$ and $Q:I$ of a $p_g$-ideal $I$}\label{s:comp}

Let $I=I_Z$ be an integrally closed ideal of $A$ represented on some resolution $X$, and let $Q$ be any minimal reduction of $I$. 
We begin by recalling a description of core when $I$ is stable.

\begin{lem}\label{l:core} 
Assume that $I^2=QI$. Then 
\begin{equation}\label{eq:core}
\core(I)=Q^2:I=(Q:I)I=(Q:I)Q.
\end{equation}
Moreover, if $\core(I)$ is integrally closed, then so is $Q:I$.
\end{lem}

\begin{proof}
The formula \eqref{eq:core} is obtained in \cite[5.6]{OWYgood}. 
Assume that $\core(I)$ is integrally closed and $Q$ is generated by $a,b\in I$.    Let  $x\in A$ be integral over $Q:I$. 
Then $ax$ is integral over $Q (Q:I)$.
Since $(Q:I) Q=(Q:I)I=\core(I)$, 
we have $ax\in Q (Q:I)$.
Thus $ax = au + bv$ for some $u, v \in Q:I$. 
Since $a (x - u) = bv$, 
we have that $x -u \in (b)\subset Q$, and thus $x \in Q:I$.
Hence $Q:I$ is integrally closed.
\end{proof}

\begin{lem}\label{l:DD}
Let $D>0$ be a cycle on $X$. Then we have the following.
\begin{enumerate}
\item $H^0(\cO_D(D))=0$.
\item $h^1(\cO_D(D))=(-D^2+K_XD)/2$.
\item $p_g(A)+h^1(\cO_D(D))=h^1(\cO_X(D))$. 
Therefore $h^1(\cO_X(D))\ge p_g(A)$ and the equality holds if and only if $h^1(\cO_D(D))=0$. 
\item If $h^1(\cO_D(D))=0$, then each connected component of
the support of $D$ contracts to a nonsingular point.
More precisely, if $E_1\le D$ is a $(-1)$-curve, $g\: X\to X'$ the contraction of $E_1$, and if $F:=g_*D\ne 0$, then $h^1(\cO_{F}(F))=0$ and $D=g^*F+nE_1$ with $n=0$ or $1$. 
\end{enumerate}
\end{lem}

\begin{proof} Although the claim (1) is stated in \cite{wahl.vanish},
 we will give a short proof for the convenience
 of the readers.  By adjunction, $K_D = (K_X + D)|_D$ 
and by duality on $D$, $h^0(\cO_D(D)) = h^1(\cO_D(K_X))$.  
Since $H^1(X,\cO_X(K_X)) \to H^1(D, \cO_D(K_X))$ is surjective and 
  $H^1(X,\cO_X(K_X))=0$ by Grauert-Riemenschneider vanishing theorem, we have the desired 
  result.
The statement (2) follows from (1) and 
 the Riemann-Roch formula: $h^1(\cO_D(D))=-\chi(\cO_D(D))=(-D^2+K_XD)/2$.
\par
The assertion (3) follows from (1) and the exact sequence
$$
0 \to \cO_X\to \cO_X(D)\to \cO_D(D)\to 0.
$$

(4) If $D$ contains no $(-1)$-curve, then 
$K_XD\ge 0$, and thus $2h^1(\cO_D(D))=-D^2+K_XD\ge -D^2>0$.
Assume that $2h^1(\cO_D(D))=-D^2+K_XD=0$ and 
$E_1$ be a $(-1)$-curve in $D$. Let $g\: X\to X'$ be the 
contraction of $E_1$ and $F=g_*D$. 
Then there exists an integer $n$ such that $D=g^*F+nE_1$.
We have 
$$
-D^2+K_XD=-F^2+K_{X'}F+n^2-n.
$$
Since $n^2-n\ge 0$ for any $n\in \Z$ and $-F^2+K_{X'}F\ge 0$ by (2), we obtain that $-F^2+K_{X'}F=n^2-n=0$.
Therefore we can inductively contract all components of $D$ 
to nonsingular points.  
\end{proof}

The following lemma is essential for our main theorem of this section.

\begin{lem}\label{Q:I;key}
  Let $I=I_Z$ be a $p_g$-ideal represented on some resolution $X$, and $Q$ a minimal reduction of $I$.
Assume that $\cO_X(-Z+Y)$ is generated for some cycle $0<Y<Z$ on $X$.
 Then the following conditions are equivalent. 
\begin{enumerate}
\item $I_ZI_{Z-Y}=Q I_{Z-Y}$.
\item $h^1(\cO_X(Y)) = p_g(A)$ and $I_{Z-Y}$ is a $p_g$-ideal.
\end{enumerate}
Note that the first condition implies $I_{Z-Y} \subset Q:I_Z$.
\end{lem}
\begin{proof}  Assume $Q=(a,b)$ and consider the exact sequence  
\[
0\to \cO_X\to \cO_X(-Z)^{\oplus 2}\to \cO_X(-2Z)\to 0,
\]
where the map $\cO_X(-Z)^{\oplus 2}\to \cO_X(-2Z)$ is defined by  
$(x,y)\mapsto ax+by$. Tensoring $\cO_X(Y)$ to this exact sequence, we obtain  
\[
0\to \cO_X(Y)\to \cO_X(-Z+Y)^{\oplus 2}\to \cO_X(-2Z +Y)\to 0.
\]
Note that $I_ZI_{Z-Y}=I_{2Z-Y}$ and $h^1(\cO_X(-Z+Y))= h^1(\cO_X(-2Z+Y))$ by \proref{p:sg}.
Now, taking the long exact sequence of this sequence, we see that $I_ZI_{Z-Y}=Q I_{Z-Y}$
if and only if  $H^1(\cO_X(Y))\to H^1(\cO_X(-Z+Y)^{\oplus 2})$ is injective, and this is also equivalent to 
that 
\[
2 h^1(\cO_X(-Z+Y)) = h^1(\cO_X(Y)) +  h^1(\cO_X(-2Z +Y)).
\]
Thus (1) is equivalent to that
$h^1(\cO_X(-Z+Y)) = h^1(\cO_X(Y))$.
On the other hand, we have 
 $h^1(\cO_X(-Z+Y)) \le p_g(A) \le h^1(\cO_X(Y))$
by \proref{p:le-cycle} and \lemref{l:DD}. Therefore, the 
condition (1) is equivalent to the equalities
\[
h^1(\cO_X(-Z+Y)) = p_g(A) = h^1(\cO_X(Y)),
\]
namely, the condition (2).  
\end{proof}

Now we can state the main theorem of this section.
 
\begin{thm}\label{Q:I} Let $I=I_Z$ be a $p_g$-ideal, where $Z$ is a 
$p_g$-cycle on a resolution $X$ of $\Spec A$.  Then $\core(I)$ and $Q:I$ are also represented on $X$. Furthermore,
we can write $Q:I= I_{Z-Y}$ and $\core(I)=I_{2Z-Y}$, 
where $Y$ is the maximal positive cycle on $X$ satisfying the following two conditions:
\begin{enumerate} 
\item $h^1(\cO_X(Y))= p_g(A)$ and 
\item $I_{Z-Y}$ is a $p_g$-ideal.
\end{enumerate}
In particular, if $I$ is a $p_g$-ideal, so are $Q:I$ and $\core(I)$.
\end{thm}

\begin{proof}  
By \proref{p:pgC}, 
the Rees algebra $\cR(I)$ is normal and Cohen-Macaulay.
Hence $\core(I)$ is integrally closed by 
\cite[Proposition 5.5.3]{Hyry-Smith}.
It follows from Lemma \ref{l:core} that $Q: I$ is also integrally closed. By \corref{c:sg}, we have $I\subset Q:I$.
Suppose that $I$ and $Q: I$ are represented by cycles $Z'$ and 
$Z'-Y$ on some resolution $X'$, respectively, where $Y$ is an effective cycle on $X'$. 
Note that $Y$ satisfies the condition (1) of \lemref{Q:I;key} by \eqref{eq:core}.
We may assume that $Z'C<0$ or $(Z'-Y)C<0$ for any $(-1)$-curve $C$ on $X'$.
If $C$ is a $(-1)$-curve and $Z'C=0$, then $YC>0$.
However, this is impossible because $Y$ blows down to a nonsingular point and $YC\le 0$ if $C\le Y$ by \lemref{l:DD}.
Thus $I$ is minimally represented on $X'$.
Hence there exists a birational morphism $\alpha\: X\to X'$ and
$Q:I$ is also represented on the resolution $X$ by $\alpha^*(Z'-Y)=Z-\alpha^*Y$.
Thus we can write $Q:I= I_{Z-Y_0}$, where $Y_0=\alpha^*Y\ge 0$. 
By \eqref{eq:core} and \lemref{Q:I;key}, 
we see that $\core(I)=I_{2Z-Y_0}$ and that 
$Y_0$ satisfies our conditions (1) and (2).
On the other hand, if a cycle $Y\ge 0$ on $X$ also satisfies our conditions, then  \lemref{Q:I;key} implies that $I_{Z-Y}\subset Q: I=I_{Z-Y_0}$.
Hence we must have $Y\le Y_0$. 
\end{proof}

Note that even if $I$ is minimally represented on $X$, $Q:I$ is not necessarily minimally represented on $X$.
In \thmref{algorithm}, 
we shall give an algorithm to calculate $Q: I$ and  a minimal good ideal containing $I$ 
starting from a $p_g$-ideal $I= I_Z$.

In the situation of \thmref{Q:I}, $I$ is a good ideal if and only if $Y=0$.
Hence we obtain the following. 

\begin{cor}\label{c:cores}
Assume that $I$ is a $p_g$-ideal.
Then the following are equivalent.
\begin{enumerate}
\item $I$ is a good ideal.
\item $\core(I)=I^2$.
\item $\core(I^2)=\core(I)^2$.
\end{enumerate}
\end{cor}

\begin{thm}\label{algorithm} Let $I=I_Z$ be a $p_g$-ideal represented on $X$ and let
$Q$ be a minimal reduction of $I$. 
\begin{enumerate} 
\item Let $X=X_1\to X_2\to \cdots \to X_{n+1}$ be the sequence of contractions, where each 
$f_i: X_i \to X_{i+1}$ is a contraction of a $(-1)$-curve $E_i$ on $X_i$, which 
does not intersect the cohomological cycle $C_{X_i}$. 
This process stops when every $(-1)$-curve on $X_{n+1}$ intersects $C_{X_{n+1}}$.
Let $K$ be the relative canonical divisor $K_{X/X_{n+1}}$ of the composite $f: X \to X_{n+1}$. 
\par
Then $Q: I = H^0(X, \cO_X( -Z +K))$. Note that $Z-K$ is anti-nef if and only if any $(-1)$-curve $C$ on $X$ satisfies $ZC<0$; in this case, $K$ coincides with $Y$ in \thmref{Q:I}. 
Also, $f_*Z$ is a 
$p_g$-cycle 
and $I_{f_*Z}$ is the minimal 
good ideal containing $I_Z$. In particular, every two-dimensional excellent normal local domain has good ideals.

\item Let $C$ be a $(-1)$-curve on $X$ which does not intersect $C_X$ such that $ZC = - m<0$ and $g: X\to X'$ the contraction of $C$.
Then $g_*Z$ is also a $p_g$-cycle.
Let $I' = I_{g_*Z}$ and $Q'$ a minimal reduction of $I'$.  In this case, if $Q': I' = I_{Z'}$ 
for some cycle $Z'$ on $X'$, then 
$Q: I = I_{g^*Z' + (m-1)C}$.
\end{enumerate}
\end{thm}
\begin{proof}
Let $F_i$ be the total transform of $E_i$ on $X$.
Then $F_iF_j=0$ for $i\ne j$.
Let 
\[
b_i=-Z F_i, \quad a_i = \min\{1, b_i\}, \quad
a=\sum_{i=1}^n a_i, \quad \text{and} \quad b=\sum_{i=1}^n b_i.
\]
We will prove that a cycle $Y := \sum_{i=1}^n a_i F_i$ satisfies
$Q:I = I_{Z-Y}$.
Since $a_i^2=a_i$, we have 
\[
-Y^2+K_XY=\sum_{i=1}^n a_i (-F_i^2+K_XF_i)
=\sum_{i=1}^n a_i (-E_i^2+K_{X_i}E_i)=0.
\]
From \lemref{l:DD}, we have $h^1(\cO_Y(Y))=0$
and that $Y$ is the maximal cycle on $X$ satisfying $h^1(\cO_X(Y))=p_g(A)$.  
By \thmref{Q:I}, it is sufficient to show that $Z-Y$ is a $p_g$-cycle.
We have $f_* Z$ and $f^*f_* Z$ are $p_g$-cycles  by \proref{p:CX}
 and $Z=f^*f_* Z+\sum _{i=1}^n b_iF_i$. 
Take any decreasing sequence of cycles $\{Z_0, \dots, Z_b\}$ 
such that
\[
Z_0=Z, \quad Z_a=Z-Y, \quad Z_b=f^*f_* Z, \quad C_i := Z_i-Z_{i+1} \in \{F_1, \dots, F_n\}.
 \]
Then every $Z_i$ is anti-nef.
Let $h_i=h^1(\cO_X(-Z_i))$. Consider the exact sequence
\[
0 \to \cO_X(-Z_i) \to \cO_X(-Z_{i+1}) \to \cO_{C_i}(-Z_{i+1}) \to 0.
\]
Note that any nef invertible sheaf $\cL$ on $C_i$ is generated and satisfies $H^1(\cL)=0$ (cf. \cite{Li}).
Since $H^1(\cO_{C_i}(-Z_{i+1}))=0$, we have $h_i\ge h_{i+1}$; 
however the equality holds for every $0\le i <b$ because $h_b=h_0=p_g(A)$.
Thus $H^0(\cO_X(-Z_{i+1})) \to H^0(\cO_{C_i}(-Z_{i+1}))$ is surjective and $h_i=p_g(A)$ for each $0\le i <b$.
Thus $H^0(\cO_{X}(-Z_{i+1}))$ has no base points on $\supp C_i$. Therefore if $\cO_X(-Z_i)$ is generated, so is $\cO_X(-Z_{i+1})$.
Since $\cO_X(-Z_0)$ is generated, 
$\cO_X(-Z_i)$ is also generated and $Z_i$ is a $p_g$-cycle for every $0\le i\le b$. 
Hence $Q:I=I_{Z-Y}$.
Now it is clear that for any minimal reduction $Q^*$ of $I_{f_*Z}$, we have $Q^*:I_{f_*Z}=I_{f_*Z}$, namely, $I_{f_*Z}$ is a good ideal.
The argument above also shows this is the minimal good ideal containing $I_Z$.
If $ZC<0$ for every $(-1)$-curve $C$ on $X$ not intersecting $C_X$, then $a_i=1$ for 
every $1\le i \le n$, and thus $Y=K$.

Let $C=E_1$, $g=f_1$, and $X'=X_2$.
We will show that if $ZC=0$, we have 
$H^0(X, \cO_X(- Z +K)) = H^0(X', \cO_{X'}(- g_*Z + K'))$, 
where $K'=K_{X'/X_{n+1}}$.
This claim implies the formula $Q:I=H^0(\cO_X(-Z+K))$; in fact, it can be reduced to the case $n=1$.
We have that $g^*\cO_{X'}(-g_*Z+K')=\cO_X(-Z+K-C)$ and the exact sequence
$$
0 \to \cO_X(-Z+K-C) \to \cO_X(-Z+K) \to \cO_C(-Z+K) \to 0.
$$  
Since $ZC=0$, we have $\cO_C(-Z)\cong \cO_C$ and $H^0(\cO_C(-Z+K))\cong H^0(\cO_{\mathbb P^1}(-1))=0$. Hence we obtain the claim.
Assume that $m:=-ZC>0$.
By \proref{p:CX}, $g_*Z$ is a $p_g$-cycle on $X'$.
Let $I'=I_{g_*Z}$ and $Q'$ a minimal reduction of $I'$.
Let $Y'$ be the effective cycle on $X'$ such that $Q':I'=I_{g_*Z-Y'}$. 
Then $Y=g^*Y'+C$. 
Thus $Z-Y=g^*g_*Z+mC-(g^*Y'+C)=g^*(g_*Z-Y')+(m-1)C$.
\end{proof}

\begin{rem}\label{r:goodcl} 
(1) If $I=I_Z$ is a $p_g$-ideal, 
then the minimal good ideal containing $I$ (the ``good closure" of $I$) is obtained by ``taking $Q:I$'' 
several times (cf. \thmref{t:subset}). Namely, if $I_1= Q: I$ and $Q_1$ is a minimal reduction of 
$I_1$, then $I_2= Q_1: I_1$ is again a $p_g$-ideal.
We obtain  $Q_m: I_m = I_m$ for some $m>0$. 
The smallest positive integer $m$ with this property is given by $m=\max\{b_1, \dots, b_{n}\}$, where $b_i$'s are the integers in the proof of \thmref{algorithm}.

(2) If $A$ is a rational singularity, then $X_{n+1}$ in 
\thmref{algorithm} is always the minimal resolution since we may regard  $C_X= 0$ (cf. \cite[5.1]{OWYgood}).
If $A$ is a rational Gorenstein singularity (i.e., rational double point), then $K=K_X$ and $Q: I$ coincides with the multiplier ideal $\cal J(I)$. 
In this case, we have $\core(I)=\cal J(I^2)=I\cal J(I)$ by \lemref{l:core}.
\end{rem}

\begin{ex}
Let $X_0$ be the minimal resolution and $f\: X \to X_0$ the natural morphism. 
Assume that the exceptional divisor $F$ of $X_0$ is a nonsingular rational curve and $F=f(E_1)$.
For any ideal $I=I_Z$ represented by $Z=\sum n_i E_i$ on $X$, 
the good closure of $I$ is $I_{f_*Z}=I_{n_1F}=\m^{n_1}$ (cf. \cite{artin.rat}). Since good ideals in a rational singularity are 
integrally closed by \cite[2.4]{OWYgood}, the set of good ideals in $A$ is $\defset{\m^n}{n\in \Z_{>0}}$.

Let us consider an explicit example of the procedure in  \remref{r:goodcl} (1).
Suppose that $A=k[[x,y,z]]/(x^2-yz)$ and $X$ is obtained by the blowing up at the nodal points of $\di_{X_0}(x)$.
We write $E=E_0+E_1+E_2$, where $E_0^2=-4$ and $E_1^2=E_2^2=-1$.
Then $K:=K_{X/X_0}=E_1+E_2$ and $E$ is the maximal ideal cycle.
Let $W=E_0+2(E_1+E_2)$. Then $I_W=(x)+\m^2$ since $\ell_A(A/I_W)=3$ by Kato's Riemann-Roch formula. 
Fix an integer $n$ and let $Z=nE_0+2n(E_1+E_2)$.
For $0\le k \le n$, let $Z_k=nE_0+(2n-k)(E_1+E_2)$. 
Since $Z_k=(n-k)W+kE$, $I_k:=I_{Z_k}=((x)+\m^2)^{n-k}\m^k$.
From \thmref{algorithm}, for any minimal reduction $Q_k$ of $I_k$, we obtain that $Q_k:I_k=I_{k+1}$ for $0\le k<n$ and $Q_n:I_n=I_n$.
\end{ex}

The next theorem follows from \thmref{algorithm}.

\begin{thm}\label{t:good_ch}
There always exist good $p_g$-ideals in any 
two-dimensional excellent normal local domain containing an algebraically closed field. 

 Let $I = I_Z$ be a $p_g$-ideal minimally represented on $X$.   Then $I_Z$ is good if and only if every $(-1)$-curve on 
$X$ intersects the cohomology cycle $C_X$.
\end{thm}

\begin{cor}
Assume that $I$ is a $p_g$-ideal.
Then
 $I$ is a good ideal if and only if 
 $\core(II')=\core(I)\core(I')$ for any $p_g$-ideal $I'$.
\end{cor}
\begin{proof}
Assume that $I$ and $I'$ are represented on $X$ by $p_g$-cycles $Z$ and $Z'$, respectively, and that $I$ is a good ideal.
By \thmref{t:good_ch}, we may assume that $Z'C<0$ for every $(-1)$-curve $C$ on $X$ which does not intersect $C_X$; then $Z+Z'$ has also this property.
Let $Q'$ (resp. $Q''$) be a minimal reduction of $I'$ (resp. $II'$).
By \thmref{algorithm} (1), we obtain that $Q'':II'=I_{Z+Z'-K}$ and $Q':I'=I_{Z'-K}$. 
Since $\core (I)=I^2=I_{2Z}$ by \corref{c:cores}, it follows from  \thmref{Q:I} that
\[
\core(II')=I_{2(Z+Z')-K}=I_{2Z+(2Z'-K)}=\core(I)\core(I').
\]

The converse follows from \corref{c:cores}.
\end{proof}

From (1) of Theorem \ref{algorithm} and \lemref{l:core}, 
we have the following result concerning Question \ref{Q1} of \S 2.

\begin{thm}\label{t:subset}  Let $I_1\subset I_2$ be $p_g$-ideals,
 and let $Q_i$ be a minimal  reduction of $I_i$ ($i=1,2$). 
Then $Q_1:I_1\subset Q_2:I_2$ and also 
$\core(I_1)\subset \core(I_2)$.
\end{thm}

\begin{quest}
Let $(A,\frm)$ be a two-dimensional excellent normal local domain and let 
$I_1\subset I_2$ be  integrally closed $\m$-primary ideals and 
let $Q_i$ be a minimal  reduction of $I_i$ for $i=1,2$. 
Assume that $I_2$ is a $p_g$-ideal (we do not assume that $I_1$ is a $p_g$-ideal).
Is  it  true that $\core(I_1)\subset \core(I_2)$ and 
that $Q_1: I_1 \subset Q_2:I_2$?
\end{quest}

\section{Existence of $p_g$-ideal containing a given element of $I$.}\label{s:pgexs}

We show that for any $g\in \m$ there exists a $p_g$-ideal containing $g$. In fact, we prove the following.

\begin{thm}
\label{t:pgf}
Let $I$ be an integrally closed $\m$-primary ideal and $g$ an arbitrary element of $I$. 
Then there exists $h\in I$ such that the integral closure of the ideal $(g,h)$ is a $p_g$-ideal. 
\end{thm}

\begin{proof}
Suppose that $I$ is represented 
by a cycle $Z>0$ on a resolution $X$. 
We write $\di_X(g)=Z_X+H_X$, where $Z_X$ is a cycle and $H_X$ is
the strict transform of the divisor $H:=\di_{\spec A}(g)$.
Let $C_X$ be the cohomological cycle on $X$.
Then, by \proref{p:CX}, $Z_X$ is a $p_g$-cycle if and only if 
\begin{equation}\label{eq:CcapH}
\supp C_X \cap \supp H_X=\emptyset.
\end{equation} 
By \proref{p:CC}, taking blowing-ups at the intersection of the cohomological cycle and the strict transform of $H$ successively, 
we may assume that the condition \eqref{eq:CcapH} is satisfied.
Since $Z_X\ge Z$, we have $I_{Z_X}\subset I$.
Since $\cO_X(-Z_{X})$ is generated, there exists $h\in I_{Z_X}$ such that $g$ and $h$ generate $\cO_X(-Z_X)$.  
Then we obtain that $I_{Z_X}=\ol{(g,h)}$.
\end{proof}

The idea for obtaining the $p_g$-ideal in the proof above is used in the last paragraph of \exref{ex:244}, where the element ``$h\in I$'' in \thmref{t:pgf} can be $g+h'$ with general element $h'\in \ol{\m^3}$.

\begin{defn}\label{d:pgass}
Let $h\in \m$. There exists a unique resolution $f\: X\to \spec A$ such that $C_X$ and the strict transform $H\subset X$ of $\di_{\spec A}(h)$ have no intersection, and that $f$ is minimal among the resolutions with these properties.
Note that such a resolution is obtained by taking blowing-ups successively from the minimal resolution.
Then $Z:=\di _X(h)-H$ is a $p_g$-cycle by \proref{p:CX}
 and every $(-1)$-curve $C$ on $X$ satisfies $ZC<0$ and 
$C_XC<0$.
We call $I_Z$ the {\em $p_g$-ideal associated with $h$}.
\end{defn}

From now on, we assume that the following condition holds for any 
integrally closed $\frm$-primary ideal $I\subset A$:
\begin{enumerate}
\item [(C1)]
If $h\in I$ is a general element and $I$ is represented by $Z$ on $X$, and if $\di _X(h)=Z+H$, then the divisor $H$ is a disjoint union of 
nonsingular curves and each component of $H$ intersects the exceptional set transversally.
\end{enumerate}
This condition holds in case the singularity is defined over a field of characteristic zero, since Bertini's theorem can be applied to the image of $H^0(\cO_X(-Z))\to H^0(\cO_{E_i}(-Z))$.

\begin{thm}\label{t:colength}
Let $J = I_M$ be an integrally closed $\frm$-primary ideal
represented by a cycle $M$ on a resolution $Y\to \spec A$, which is 
obtained by taking the minimal resolution of the blowing-up of
 $J$.
Assume that $h\in J$ satisfies the condition 
$\di_Y(h)=M+H$, where $H$ is the strict transform of 
$\di_{\Spec(A)}(h)$ on $Y$.

If $I$ is the $p_g$-ideal associated to $h$, 
then we have 

\[
e(I)-e(J)\ge C_Y H = -C_YM  \quad \text{and} 
\quad \ell_A(J/I) \le -C_YM,
\]
where the latter inequality is strict if $J$ is not a 
$p_g$-ideal.

If, moreover, $h\in J$ is a general element, then 
we have the equalities 
\[e(I)-e(J) = C_Y H = -C_YM, \quad 
\ell_A(J/I) = C_Y H - \left(p_g(A) - h^1(\cO_Y(-M))\right)
\]
and in this case $I$ is a maximal $p_g$-ideal 
contained in $J$. 

In particular, every $p_g$-ideal is associated with its general element.
\end{thm}

\begin{proof} First note that by definition, $J$ is a $p_g$-ideal if $p_g(A)-h^1(\cO_Y(-M))=0$, and this condition is equivalent to that $C_YM=0$ by \proref{p:CX}.

We prove $e(I)-e(J)\ge C_Y H = -C_YM$  by induction on 
$C_Y H = -C_YM$. Our assertion is obvious if  $C_Y H = 0$.

In order to construct a resolution $X \to Y$ such that $I = I_Z$ is represented by a cycle $Z$ on $X$,
we proceed by successive blowing-ups at points 
\[X= Y_r \to Y_{r-1} \to \ldots \to Y_1\to Y_0=Y\]
and take anti-nef cycles  $Z_0=M, Z_1,\ldots , Z_r =Z$ 
so that 
\[\di_{Y_i}(h) = Z_i + H_i,\]
where $H_i$ is the strict transform of $H$ on $Y_i$. 
Note that $C_X H = 0$. 

We assume $C_Y H > 0$ and take a point $P\in C_Y\cap H$.
We also assume that $H$ has multiplicity $a$ at $P$.
Let $\pi : Y_1\to Y$ be the blowing-up at $P$ and put 
$E_0 = \pi^{-1}(P)$.
Then $Z_1 = \pi^*(M) + aE_0$, 
since $\di_{Y_1}(h) = \pi^*(\di_Y(h)) =\pi^*(M) + aE_0 + H_1
=Z_1 + H_1$.
We have seen in Proposition \ref{p:CC} that 
$C_{Y_1} = \pi^*(C_Y) - E_0$. Since $C_{Y_1}H_1
= C_YH -a < C_Y H$, we may assume $-Z^2 - (-Z_1)^2 \ge 
C_{Y_1}H_1$ by our induction hypothesis.  
Also, we have $(-Z_1^2) - (-M^2) = a^2$.  Hence we have
\[(-Z_1^2) - (-M^2) \ge C_Y H - C_{Y_1}H_1 \]
with equality if and only if $a=1$. 
 \par
If $h\in J$ is a general element, by the condition (C1), we have $a=1$ at every point 
$P\in C_Y\cap H$ and hence equality $e(I) - e(J) = C_Y H$.
Furthermore, from the exact sequence 
\[ 0 \to \cO_{Y_1}( - Z_1) \to \cO_{Y_1} ( - \pi^*(M)) \to 
\cO_{E_0} \to 0,\]
we have $I_{1}:=H^0(Y_1, \cO_{Y_1}(-Z_1)) \subset I_M$ with $\ell_A( I_M/I_1)\le 1$ and 
$h^1(Y_1, \cO_{Y_1}(-Z_1)) = h^1( Y, \cO_{Y}(-M))$ or 
$h^1(Y_1, \cO_{Y_1}(-Z_1)) = h^1( Y, \cO_{Y}(-M))+1$,   the 
latter holds if and only if $I_1 =  I_M$.  
Hence we have the equality 
\[
\ell_A(J/I) = C_Y H - \left(p_g(A) - h^1(\cO_Y(-M))\right).
\]

Now, assume $e(I) - e(J)= C_Y H$ and  $I'$ is a $p_g$-ideal such that 
$I\subset I'\subset J$.  We may assume that $I=I_Z$ and $I'= I_{Z'}$  with $Z\ge Z'$ 
are represented on a resolution $X$ and we have a morphism $\phi\: X \to Y$.  Since 
$I'\subset J$, we have $Z'\ge \phi^*(M)$. 
Since $M=\phi_*(Z) \ge \phi_*(Z') \ge \phi_*(\phi^*(M))=M$, 
we must have $\phi_*(Z')= M$. \par
Take a general element $h'\in I'$ so that $\di_X(h') = Z'+ \tilde{H}'$, where $\tilde{H}'$ is the strict transform of 
$\di_{\Spec(A)}(h')$ on $X$.
Then we have $\di_Y(h') = \phi_*(\di_X(h')) = M + H'$, where $H'= \phi_*(\tilde{H}')$, 
and we have $C_YH'= - C_Y M$. Since $I'$ is a $p_g$-ideal, we must 
have $C_{X}\tilde{H}'=0$. By the argument as above,  we have 
$e(I') - e(J) \ge C_Y H'=-C_YM = e(I) - e(J)$.  Since $I \subset I'$, we have $I'=I$ and 
we have shown that $I$ is a maximal $p_g$-ideal contained in $J$.
\end{proof}

\begin{cor}
Let $I$ be a maximal $p_g$-ideal of $A$.  Then $I$ is a good ideal with 
$e(I) = e(\m) - C_X M_X$
and 
$\mu_A(I):=\ell_A(I/I\m)=e(\m)+1$, where $f \: X\to \Spec A$ is any resolution such that $\m$ and $I$ are represented on $X$ and $\m \cO_X=\cO_X(-M_X)$.
\par
Furthermore if $A$ is Gorenstein, then 
$$
\ell_A(A/I)=e(I)/2=e(\m) + p_a(M_X)-1.
$$
\end{cor}

\begin{proof}
We use the notation above and assume that $J=I_M=\m$.
By \remref{r:goodcl} (1), $I$ is a good ideal.
Let  $\phi\: X \to Y$ be the natural morphism.

Since $\phi^*M=M_X$ and $\phi^*C_Y=C_X+D$ for some cycle $D\ge 0$ with $\phi_*(D)=0$ by \proref{p:CC}, we have 
 $C_YM=(C_X+D)(\phi^*M)=C_X(\phi^*M)=C_XM_X$.
Hence the first equality follows from \thmref{t:colength}.
Since $Z=\phi^*M+D'$ with a  cycle $D'\ge 0$ satisfying $\phi_*(D)=0$
 and $Z M_X=(\phi^*M)^2=M^2$, it follows from \cite[6.2]{OWYgood} that $\mu_A(I)=-M^2+1=e(\m)+1$.

Suppose that $A$ is Gorenstein.
Then we may assume that $K_X$ and $K_Y$ are cycles.
It follows from \remref{r:Ggood} and \cite[4.2]{OWYgood} that  $\ell_A(A/I)=-Z^2/2$ and $K_X Z=0$.
By \proref{p:CC}, 
there exists an effective cycle $C'$ such that $K_Y=-C_Y+C'$ 
and $\supp C_Y \cap \supp C' = \emptyset$, 
and $K_X+C_X=\phi^*(K_Y+C_Y)$.
Then we obtain that
\[
(K_Y+C_Y)M=\phi^*(K_Y+C_Y)Z=(K_X+C_X)Z=0.
\]
Thus $C_Y M=-K_Y M$. 
Using the formulas above and $p_a(M)=(M^2+K_YM)/2+1$, we obtain 
\[
e(I)=-M^2+K_Y M=2(-M^2)+M^2+K_Y M=2(e(\m)+p_a(M)-1).
\]
Clearly, $p_a(M)=p_a(M_X)$, since $\phi^*M=M_X$. 
Hence we obtain the last assertion. 
\end{proof}

\begin{ex}
Let $C$ be a nonsingular curve with genus $g$ and $D$ a divisor on $C$ with $\deg D=e>0$.
Let $R=\bigoplus_{n\ge 0}R_n$, where $R_n=H^0(\cO_C(nD))$, 
 and $\m=\bigoplus_{n\ge 1}R_n$. 
Let $a=a(R)$, where 
$$
a(R)=\max\defset{n\in \Z}{H^1(\cO_C(nD))\ne 0} \ \ \text{(cf. \cite[(2.2)]{KeiWat-D})}.
$$ 
Assume that $a\ge 0$, $aD\sim K_C$, and that $R$ is generated by $R_1$ as $R_0$-algebra. 
Then $R$ is Gorenstein by \cite[(2.9)]{KeiWat-D}.
Let $Y\to \spec R$ be the blowing up of $\m$ and $E\subset Y$ the exceptional set; we have $E\cong C$.
Then $Y$ is the minimal resolution with $E^2=-e$ and $\m$ is represented by $M:=E$.

Since $R$ is Gorenstein, there exists an integer $k$ such that $K_Y=-kE$. From $(K_Y+E)E=\deg K_C=ae$, we have $K_Y=-C_Y=-(a+1)E$.
We have 
\[
-M^2=e, \quad C_YM=(a+1)e.
\]
Let $I$ be a $p_g$-ideal associated with a general element $h\in \m$. Since $p_a(M)=g$ and $h^1(\cO_C((a+1)D))=0$,
from \thmref{t:colength}, we obtain that 
$$
\mu_R(I)=e+1, \quad  \ell_R(R/I)=e+g-1=h^0(\cO_C((a+1)D)).
$$

We will show that $I=\m^{a+2}+(h)$.
Let $X\to \spec R$ ($A$ is replaced by $R$), $Z$, and $H$ be as in \defref{d:pgass}, and let  $\phi\: X\to Y$ be the natural morphism.
To obtain $X$ from $Y$, we need $a+1$ blowing ups at the intersection of the exceptional set and each irreducible component of the strict transform of $\di_{\spec R}(h)$ (cf. the proof of \thmref{t:pgf}).
Since the coefficient of the irreducible component of the exceptional set of $X$ intersecting $H$ in the cycle $Z$ 
is $a+2$, we have  $(a+2)\phi^*M\ge Z$.
Therefore, $\m^{a+2}+(h)\subset I$.
Note that $R_n\cong \m^n/\m^{n+1}$ and $\ell_R(R/\m^{a+2})=\sum_{n=0}^{a+1}\dim R_n$.
Since $h\in \m\setminus \m^2$ and $\bigoplus _{n\ge 0}\m^n/\m^{n+1}$ is a domain, 
the homomorphism
$$
R/\m^{a+1} \xrightarrow{\ \ \times h \ \ } (h)+\m^{a+2}/\m^{a+2}
$$
is bijective. Therefore 
\[
\ell_R(R/(h)+\m^{a+2}) =\ell_R(\m^{a+1}/\m^{a+2}) =\ell_R(R_{a+1})=h^0(\cO_C((a+1)D))
=\ell_R(R/I).
\]
Hence $I=(h)+\m^{a+2}$.
\end{ex}

\begin{acknowledgement} 
 The authors thank the referee for careful reading of the paper and 
helpful comments.
\end{acknowledgement}


\providecommand{\bysame}{\leavevmode\hbox to3em{\hrulefill}\thinspace}
\providecommand{\MR}{\relax\ifhmode\unskip\space\fi MR }
\providecommand{\MRhref}[2]{%
  \href{http://www.ams.org/mathscinet-getitem?mr=#1}{#2}
}
\providecommand{\href}[2]{#2}

\end{document}